%% file: Galois_Constr.tex
\documentclass[a4paper,11pt]{amsart}
\usepackage{amssymb, verbatim, amsmath, amsthm}

\newtheorem{definition}{Definition}
\newtheorem{theorem}[definition]{Theorem}
\newtheorem{example}{Example}
\newtheorem{algo}{Algorithm}
\newcommand{\Q}{\mathbb{Q}}
\def\stab{\mathop{\mathrm{Stab}}\nolimits}
\def\Stab{\mathop{\mathrm{Stab}}\nolimits}
\def\Gal{\mathop{\mathrm{Gal}}\nolimits}
\def\GGal{\mathop{\mathrm{GeoGal}}\nolimits}
\def\Sym{\mathop{\mathrm{Sym}}\nolimits}
\def\deg{\mathop{\mathrm{deg}}\nolimits}
\def\disc{\mathop{\mathrm{disc}}\nolimits}
\newcommand{\F}{\mathbb{F}}
\newcommand{\C}{\mathbb{C}}
\newcommand{\R}{\mathbb{R}}
\newcommand{\Z}{\mathbb{Z}}
\newcommand{\regsym}{\textsuperscript{\tiny\textregistered}}
\newcommand{\trademark}{\textsuperscript{\tiny\texttrademark}}

\begin{document}

\title{Constructions using Galois Theory}
\author{Claus Fieker}
\address{Fachbereich Mathematik, Universit\"at Kaiserslautern, Germany}
\email{fieker@mathematik.uni-kl.de}
\author{Nicole Sutherland}
\address{Computational Algebra Group, School of Mathematics and Statistics, University of Sydney, Australia}
\email{nicole.sutherland@sydney.edu.au}
\date{}
\keywords{Galois Groups, Fixed Fields, Splitting Fields, Radical extensions}
\subjclass[2020]{Primary 11R32 11Y40 12F10; Secondary 12F12}

\begin{abstract}
We describe algorithms to compute fixed fields, 
splitting fields and towers of radical extensions 
without using polynomial factorisation in towers
or constructing any field containing the splitting 
field, instead extending Galois group computations for this task.
We also describe the computation of geometric Galois groups (monodromy groups)
and their use in computing absolute factorizations.
\end{abstract}
\maketitle

\section{Introduction}

This article discusses some computational applications of Galois groups. 
The main Galois group algorithm we use has been previously discussed 
in~\cite{fieker_klueners} for number fields 
and~\cite{suth_galois_global, hilbert} for function fields.
These papers, respectively, detail an algorithm which has no degree restrictions on input
polynomials and adaptations of this algorithm for polynomials over function 
fields. Some of these details are reused in this paper and we will refer to the 
appropriate sections of the previous papers when this occurs.

As the Galois group algorithms used are not theoretically degree restricted,
neither are the algorithms we describe here 
for computing splitting fields. Most Galois
group algorithms developed prior 
to~\cite{fieker_klueners, suth_galois_global, hilbert}
used tabulated information and were limited in both theory and practice.

Galois Theory began in an attempt to solve polynomial equations by 
radicals. As Galois groups in general correspond to splitting fields, it is 
worthwhile to consider how the computation of a Galois group can aid the 
construction of such fields.
In particular, in solving a polynomial by radicals we 
compute a tower of radical extensions as a splitting field.
Algorithm~\ref{tower} constructs splitting fields 
as towers of extensions using the Galois group and
can also be used to compute a tower of radical extensions if the 
Galois group of a given polynomial is soluble (Algorithm~\ref{radical_algo}).

A splitting field as a single extension can be computed as
a fixed field (Algorithm~\ref{fixed}), however, we go further in stating an algorithm which computes a
splitting field as a tower of extensions of the coefficient
field of a polynomial in Section~\ref{sect_split}.
The computation of the defining polynomials for the extensions in this tower,
described in Algorithm~\ref{next}, is the main contribution of 
Section~\ref{sect_split} and indeed this paper.
These defining polynomials are also proved to be correct (Theorem~\ref{th_next})
as well as the 
splitting fields which are defined by them (Theorem~\ref{th_split}).

The use of
Algorithm~\ref{tower} to compute 
a tower of radical extensions is described in Section~\ref{sect_radical}.
First though, we consider the simpler computation of fixed fields of subgroups 
of Galois groups in Section~\ref{sect_fixed}.

Another Galois group computation is that of computing a geometric Galois group, 
the Galois group of a polynomial $f$ over $\Q(t)$ considered as a polynomial 
over $\C(t)$, discussed in Section~\ref{sect_geom}. This is also known as the monodromy group of the bi-variate polynomial with applications in algebraic geometry.
Geometric Galois groups help us compute the absolute factorization
of a polynomial over $\Q(t)$ (Section~\ref{sect_abs}).

The computations mentioned are available in {\sc Magma}~\cite{magma225}
for some coefficient rings. Examples in this paper are illustrated using {\sc Magma}.

We begin with some definitions for easy reference.

\begin{definition}
A {\em splitting field} $S_f$ of a polynomial $f \in F[x]$ is a field, of minimal degree over the field $f$, over which
$f$ can be factorized into linear factors.
\end{definition}

\begin{definition}
The {\em Galois group}, $\Gal(f)$, of a polynomial $f$ over a field $F$ 
is the automorphism group 
of $S_f/F$ where $S_f$ is a splitting field of $f$ over $F$.
Elements of $\Gal(f)$ permute the roots of $f$.
\end{definition}

Invariants and resolvents are an important part of our algorithms.
Invariants are discussed extensively in~\cite{fieker_klueners, suth_galois_global}.  Let $R$ be a commutative unitary domain and 
$I(x_1, \ldots x_n) \in R[x_1, \ldots x_n]$. A permutation $\tau \in \Sym(n)$
acts on $I$ by permuting $x_1 \ldots, x_n$ and we write $I^{\tau}$ for this 
action. 

\begin{definition}
\label{invar_defn}
A polynomial $I(x_1, \ldots, x_n) \in R[x_1, \ldots, x_n]$
such that $I^{\tau} = I$ for all $\tau \in H$ for some group $H \subseteq S_n$
is said to be {\em $H$-invariant}.

When $H \subset G$, a $H$-invariant polynomial $I(x_1, \ldots, x_n) \in R[x_1, \ldots, x_n]$
is a {\em $G$-relative $H$-invariant} polynomial if $\stab_G I = H$.

For a $G$-relative $H$-invariant polynomial $I$,
$$Q_{(G, H)}(y) = \prod_{\tau \in G // H} (y - I^{\tau}(x_1, \ldots, x_n)),$$
is a $G$-relative $H$-invariant {\em resolvent polynomial}
where $G // H$ denotes a {\em right transversal}, 
a system of representatives for the right cosets $H \tau$.
\end{definition}

To ensure, when necessary, that a resolvent polynomial has a
root in $F$ which is a single root, as used in~\cite[Theorem 5]{stauduhar},
a suitable {\em Tschirnhausen transformation} can be applied to all variables of
the invariant we are using, giving
a change of variable (\cite{tignol} Section 6.4). 
For more information see~\cite[Section 3.7]{suth_galois_global}.

\section{Galois Group Computation}
All our algorithms here depend on the following infrastructure constructed to
compute Galois groups in the first place. Here we summarize the results
we are going to use: Starting with a polynomial $f \in F[x]$ for a field
$F$ where we can do computations, we assume
\begin{enumerate}
  \item an oracle to compute approximations to the roots of $f$ in a 
    constructive field, e.g. for $F = \Q$, this will be an unramified
    extension of a $p$-adic field, for $F=\F_q(t)$ this will
    be a Laurent-series field over an extension of $\F_q$.
    In general, this is a completion at a finite place.
  \item an oracle that bounds the size of all roots in a ``canonical'' field
    corresponding to the completions at the infinite places
  \item an oracle to ``lift'' or reconstruct elements in $F$ from
    elements in the finite completions and a bound on the 
    size in the infinite completions.
  \item given any subgroup $U$ of the Galois group and any maximal subgroup
    $H$ of $U$, find a $U$-relative $H$-invariant with integral coefficients.
  \item given any invariant find a bound on the size of the evaluation
    at the roots in the infinite completions in any ordering
\end{enumerate}
Those oracles are readily available from the implementation in {\sc Magma}
and described in~\cite{fieker_klueners, suth_galois_global, hilbert}.
\section{Fixed Fields}
\label{sect_fixed}
The Fundamental Theorem of Galois Theory displays a correspondence between 
subgroups of Galois groups and their fixed fields.
We state the following algorithm to 
compute fixed fields of subgroups $U$ of these Galois groups.

The computation is similar to the computation of the Galois group. We compute an invariant and roots to some 
useful precision and from this we can compute a polynomial with at least one root fixed by $U$
and with degree that of the index of $U$ in $G$.
If $U$ is smaller than the Galois group then none of the roots will lie in the
coefficient ring of $f$.
This polynomial will define the fixed field of the given subgroup 
and can be mapped back to be over the coefficient ring of the original polynomial.

\begin{algo}[Compute a Fixed Field of a subgroup of a Galois group (\cite{fieker_klueners_impl})]
\label{fixed}
Given a subgroup $U \subseteq G = \Gal(f), f \in F[x]$,
where $F$ may be $\Q, \Q(\alpha), \F_q(t), \F_q(t)(\alpha)$ or $\Q(t)$,
compute the subfield of $S_f/F$ fixed by $U$.

\begin{enumerate}
\item Compute a $G$-relative $U$-invariant polynomial $I$.
\item Compute a right transversal $G//U$.
\item\label{fixed_bound} Compute a bound $B$ on the evaluation of $I$ at the roots
$\{r_i\}_{i=1}^n$ and compute the roots to a precision that allows
the bound $B$ to be used.
\item Compute the polynomial $g$ with roots
$\{I^{\tau}(r_1, \ldots, r_n) : \tau \in G//U\}.$
\item\label{fixed_map} Map the coefficients of $g$ back to the coefficient ring of $f$ using $B$. The
resulting polynomial defines the fixed field of $U$.
\end{enumerate}
\end{algo}

For more information on Step~\ref{fixed_bound} and~\ref{fixed_map} see~\cite[Sections 3.8 and 3.2.1]{suth_galois_global}.

Note, here and in subsequent algorithms, we reconstruct a polynomial
from its roots. Given that the roots are only known in approximation, we need
to find a size--bound in order to deduce a precision so that the above oracles
can recover the exact data. Using the fact that the coefficients are
elementary symmetric functions in the roots, bounds are immediate. The
bounds depend on the degree and, in characteristic zero, the number of monomials
to be added. 

\begin{theorem}
\label{th_fixed}
Algorithm~\ref{fixed} computes the fixed field of a subgroup 
$U \subseteq \Gal(f), f \in F[x]$ where $F$ may be $\Q, \Q(\alpha), \F_q(t), \F_q(t)(\alpha)$ or $\Q(t)$.
\end{theorem}

\begin{proof}
Let $U^F$ be the fixed field of $U$.
Each element $\gamma \in S_f$ can be written as some polynomial in the roots of $f$, that
is, as $\gamma = I(r_1, \ldots, r_n)$ for some polynomial $I$. 
Let $\sigma \in U$, then 
$\sigma(\gamma) = I^{\sigma}(r_1, \ldots, r_n)$ so that
$\sigma(\gamma) = \gamma$ and $\gamma \in U^F$ when $I$ is $U$-invariant.
If $I$ is not $U$-invariant then $I \not= I^{\sigma}$ and 
$\sigma(\gamma) \not= \gamma$ and so $\gamma \not\in U^F$.

Now let $I$ be a $G$-relative $U$-invariant polynomial and
$U^F \ni \beta = I^{\tau}(r_1, \ldots, r_n)$ be a root of $g$. 
Since $I$ is $U$-invariant at least the identity coset will correspond to
an element fixed by $U$.
Either $\beta$
generates $U^F$ and $g$ is its minimal polynomial and is therefore
irreducible or $\beta$ generates a subfield of $U^F$ and $g$ is reducible. 
Suppose $\beta$
generates a subfield of $U^F$, then $\beta$ is an element of that subfield which
is contained in $U^F$. But this subfield will be a fixed field for a group 
containing $U$ and $I$ will be invariant outside of $U$ and hence not a $G$-relative invariant. Therefore, $\beta$ does not generate a subfield of $U^F$ and 
generates $U^F$ itself.

Theorem 4 of~\cite{stauduhar} guarantees that $g$ can be mapped back to $F$.

The correct bound and precision as discussed in~\cite[Section 3.8]{suth_galois_global}
ensures the accuracy of the polynomial $g$.

\end{proof}

\hrule
\begin{example}
\label{eg_fixed_field}
We see that a defining polynomial for a fixed field can be computed 
fairly quickly. 

\begin{verbatim}
> Fqt<t> := FunctionField(GF(101));
> P<x> := PolynomialRing(Fqt);
> f := x^6 + 98*t*x^4 + (2*t + 2)*x^3 + 3*t^2*x^2 +
>           (6*t^2 + 6*t)*x + 100*t^3 + t^2 + 2*t + 1;
> time G, _, S := GaloisGroup(f); G;                               
Time: 0.280
Permutation group G acting on a set of cardinality 6
Order = 12 = 2^2 * 3
    (2, 3)(5, 6)
    (1, 2)(4, 5)
    (1, 4)(2, 5)(3, 6)
> subg := NormalSubgroups(G : IsTransitive);
\end{verbatim}
We take a normal subgroup of $G$ in a conjugacy class with $C_6$.
\begin{verbatim}
> TransitiveGroupDescription(subg[1]`subgroup);
C(6) = 6 = 3[x]2
> time Polynomial(Fqt, a), b where 
>              a, b := GaloisSubgroup(S, subg[2]`subgroup);
x^2 + 93*t^3 + 85*t^2 + 93*t
((((x2 + x5) - (x3 + x6)) * (((x1 + x4) - (x2 + x5)) * 
((x1 + x4) - (x3 + x6)))) * ((x1 + (x2 + x3)) - (x4 + (x5 + x6))))
Time: 0.010
\end{verbatim}
\end{example}
\hrule

\subsection{Galois Quotients}
\label{sect_quot}

For any subgroup $U$ of $G = \Gal(f)$, $G$ naturally acts on $G/U$.
The image $Q$ under this action is the Galois group of the fixed field of $U$.
Now given a (potential) quotient group $Q$, {\sc Magma} is able to find all
homomorphisms of $G$ onto $Q$, thus all possible subgroups $U$ inducing them.
The fields fixed by $U$ will have Galois group isomorphic to $Q$.

\begin{enumerate}
\item
computing all subgroups of $\Gal(f)$ with order equal to the quotient of
the order of $\Gal(f)$ and the degree of $Q$.
\item determining for which of these subgroups $U$
the permutation action of $\Gal(f)$ on the cosets $\Gal(f)/U$ is isomorphic 
to $Q$ and 
\item applying Algorithm~\ref{fixed} to these subgroups.
\end{enumerate}

\bigskip
\hrule
\begin{example}
We illustrate the use of the quotients by continuing Example~\ref{eg_fixed_field}.
\begin{verbatim}
> time GaloisQuotient(S, quo<G | subg[1]`subgroup>);
[
    $.1^2 + 93*t^3 + 85*t^2 + 93*t,
    $.1^2 + 92*t,
    $.1^2 + 11*t^2 + 22*t + 11
]
Time: 0.010
> [IsIsomorphic(quo<G | subg[1]`subgroup>, GaloisGroup(g)) : 
>                                                      g in $1];
[ true, true, true ]
> time GaloisQuotient(S, quo<G|>);
[
    x^6 + (12*t + 24)*x^5 + (60*t^2 + 38*t + 38)*x^4 + (59*t^3 + 
        49*t^2 + 98*t + 66)*x^3 + (38*t^4 + 97*t^3 + 70*t^2 + 
        52*t + 79)*x^2 + (91*t^5 + 50*t^4 + 62*t^3 + 83*t^2 + 
        79*t + 89)*x + 64*t^6 + 73*t^5 + 21*t^4 + 64*t^3 + 
        23*t^2 + 38*t + 30,
    x^6 + 98*t*x^4 + (99*t + 99)*x^3 + 3*t^2*x^2 + (95*t^2 + 
        95*t)*x + 100*t^3 + t^2 + 2*t + 1
]
Time: 0.000
> [IsIsomorphic(quo<G | >, GaloisGroup(g)) : g in $1];                
[ true, true ]
\end{verbatim}
\end{example}
\hrule

\section{Splitting Fields}
\label{sect_split}
There are at least three ways a field over which a polynomial splits can be constructed from a known Galois
group. The field fixed by only the trivial subgroup will be a splitting field.
of minimal degree.
Using Algorithm~\ref{fixed} above to compute a fixed field,
a splitting field will be computed as
a primitive extension of the coefficient field of the polynomial. 

We can also compute a splitting field as a tower of algebraic extensions of the 
coefficient field. We first compute a chain of subgroups from the Galois group
and compute their
fixed fields but we need to find defining polynomials over a field in the 
tower rather than over the coefficient field of the input polynomial.

The solution of a polynomial by radicals discussed in Section~\ref{sect_radical} gives a 
third way to construct a field in which a polynomial splits --- although not of minimal
degree for the given polynomial due to the necessity of adjoining roots of unity.
It will be a splitting field for a polynomial of larger degree.

In each approach, polynomials $g$ with roots
$\{I^{\tau}(r_1, \ldots, r_n) : \tau \in H//U\}$ for some subgroups $H$ and $U$ 
of a Galois group are computed. For the latter two approaches the computation of these
defining polynomials is achieved by applying Algorithm~\ref{next}.

Before we discuss our approach we will give a small example of computing a
splitting field using repeated factorization over extensions so we can see what 
we are improving on and also an example of using the fixed field of the trivial 
subgroup to compute a splitting field as a single extension.

\bigskip
\hrule
\begin{example}
\label{eg_fact}
We continue with Example~\ref{eg_fixed_field}.
\begin{verbatim}
> tt := Cputime(); 
> F := ext<Fqt | f>; _<y> := PolynomialRing(F);
> time fact := Factorization(Polynomial(F, f)); fact;
[   
    <y + 100*F.1, 1>,
    <y + (26*t*a^5 + (66*t + 66)*a^4 + 48*t^2*a^3 + (34*t^2 + 
        34*t)*a^2 + (28*t^3 + 64*t^2 + 27*t + 64)*a + (6*t^3 + 
        6*t^2))/(t^3 + 2*t^2 + 4*t + 2), 1>,
    <y^2 + ((13*t*a^5 + (33*t + 33)*a^4 + 24*t^2*a^3 + (17*t^2 
        + 17*t)*a^2 + (66*t^3 + 35*t^2 + 70*t + 35)*a + (3*t^3 + 
        3*t^2))/(t^3 + 2*t^2 + 4*t + 2))*y + ((68*t + 68)*a^5 
        + 38*t^2*a^4 + (9*t^2 + 9*t)*a^3 + (76*t^3 + 70*t^2 + 
        39*t + 70)*a^2 + (75*t^3 + 75*t^2)*a + (89*t^4 + 15*t^3 
        + 30*t^2 + 15*t))/(t^3 + 2*t^2 + 4*t + 2), 1>,
    <y^2 + ((62*t*a^5 + (2*t + 2)*a^4 + 29*t^2*a^3 + (50*t^2 + 
        50*t)*a^2 + (8*t^3 + 4*t^2 + 8*t + 4)*a + (92*t^3 + 
        92*t^2))/(t^3 + 2*t^2 + 4*t + 2))*y + ((2*t + 2)*a^5 + 
        13*t^2*a^4 + (27*t^2 + 27*t)*a^3 + (24*t^3 + 4*t^2 + 8*t 
        + 4)*a^2 + (23*t^3 + 23*t^2)*a + (65*t^4 + 45*t^3 + 
        90*t^2 + 45*t))/(t^3 + 2*t^2 + 4*t + 2), 1>
] 
Time: 0.020
> FF<b> := ext<F | fact[3][1] : Check := false>;
> _<z> := PolynomialRing(FF);
> time Factorization(Polynomial(FF, DefiningPolynomial(FF)));
[   
     <z + 100*b, 1>,
     <z + b + (13*t*a^5 + (33*t + 33)*a^4 + 24*t^2*a^3 + (17*t^2 
         + 17*t)*a^2 + (66*t^3 + 35*t^2 + 70*t + 35)*a + (3*t^3 
         + 3*t^2))/(t^3 + 2*t^2 + 4*t + 2), 1>
] 
Time: 2.050
> time Factorization(Polynomial(FF, $2[4][1]));
[
    <z + 100*b + (75*t*a^5 + (35*t + 35)*a^4 + 53*t^2*a^3 + 
        (67*t^2 + 67*t)*a^2 + (72*t^3 + 35*t^2 + 70*t + 35)*a + 
        (95*t^3 + 95*t^2))/(t^3 + 2*t^2 + 4*t + 2), 1>,
    <z + b + (88*t*a^5 + (68*t + 68)*a^4 + 77*t^2*a^3 + (84*t^2 + 
        84*t)*a^2 + (37*t^3 + 70*t^2 + 39*t + 70)*a + (98*t^3 + 
        98*t^2))/(t^3 + 2*t^2 + 4*t + 2), 1>
]
Time: 2.050
> Cputime(tt);
4.130
\end{verbatim}
\end{example}
\hrule

\begin{example}
\label{eg_fixed_split}
In contrast, we see here the splitting field computed much quicker 
as a direct extension of
$F_{101}(t)$ which is the coefficient field of the polynomial. 
\begin{verbatim}
> time G, _, S := GaloisGroup(f); G;
Time: 0.110
Permutation group G acting on a set of cardinality 6
Order = 12 = 2^2 * 3
    (2, 3)(5, 6)
    (1, 2)(4, 5)
    (1, 4)(2, 5)(3, 6)
> time FunctionField(GaloisSubgroup(S, sub<G | >));
Algebraic function field defined over Univariate rational
function field over GF(101) by
x^12 + 47*t*x^10 + 3*t^2*x^8 + (65*t^3 + 54*t^2 + 7*t
    + 54)*x^6 + (41*t^4 + 18*t^3 + 36*t^2 + 18*t)*x^4 +
    (14*t^5 + 61*t^4 + 21*t^3 + 61*t^2)*x^2 + 80*t^6 +
    77*t^5 + 75*t^4 + 64*t^3 + 31*t^2 + 88*t + 22
Time: 0.030
\end{verbatim}
This total computation time of 0.14s is one fifteenth of the time taken
using {\tt Factorization}.
\smallskip
\hrule
\end{example}

\bigskip
Now we present an algorithm to compute a tower of fixed fields of a polynomial 
given a chain of subgroups of its Galois group. This algorithm, which calls
Algorithm~\ref{next} to compute defining polynomials, can be used to 
compute a splitting field when called by Algorithm~\ref{split} 
as well as a tower of radical extensions when
called by Algorithm~\ref{radical_algo}.

\begin{algo}[A splitting field of a polynomial using a chain of subgroups of its Galois group]
\label{tower}
Given 
\begin{itemize}
\item a polynomial $f \in F[x]$, where $F$ is $\Q, \Q(\alpha)$ or $\F_q(t)$, 
of degree $n$ with $G = \Gal(f)$, and 
\item a chain $C$ of subgroups of $G$ ending with the trivial subgroup with 
\item corresponding invariants in $I$,
\end{itemize}
compute a splitting field for $f$, of degree $\#G$ over $F$,
as a tower of algebraic extensions of $F$.
\begin{enumerate}
\item 
Set $K = F, B = \{ 1 \}$ and $P = \{ \mathop{\mathrm{Id}}(G) \}$.
\item\label{subfield_tower} 
for each $C_k \not= G$ in the chain $C$ in descending order find the
minimal polynomial of a relative primitive element by
\begin{enumerate}
\item Compute a right transversal $T_k = C_{k-1} // C_k$.
\item\label{resolve} Compute the resolvent polynomial 
$$g = \prod_{\tau \in T_k} (x - I_k^{\tau}(T(r_1), \ldots, T(r_n))) \in K[x]$$ 
using Algorithm~\ref{next} with input $f, K, T_k, I_k, B$ and $P$ to
define the next extension $K \leftarrow K[x]/g(x)$ in the towerwhere $T$
is a Tschirnhausen transformation~(\cite[Section 3.7]{suth_galois_global}).
\item
Set $B \leftarrow [ \alpha^i b : b \in B, 1 \le i < \#T_k ]$ 
where $\alpha$ is a root of $g$, a primitive element of
the extension in the tower defined by the resolvent polynomial in Step~\ref{resolve}.
\item
Set $P \leftarrow [ \tau \pi : \pi \in P,  \tau \in T_k]$.
\end{enumerate}
\item return $K$ as the splitting field of $f$ corresponding to $C$.
\end{enumerate}
\end{algo}

Algorithm~\ref{fixed} Step~\ref{fixed_map} maps elements of the local splitting
field back to $F$, as done by the computation of the Galois group,
\cite[Section 3.2.1]{suth_galois_global}. 
However, when computing a tower of extensions
the defining polynomial will be over the current top field in the tower so 
the coefficients need to be expressed with respect to an absolute basis to 
allow those coefficients to be mapped into $F$ by our oracles. This can be achieved using
Algorithm~\ref{next}.

Before considering the details of computing the resolvent polynomial, 
we first describe the preparation of the chain of subgroups
necessary to compute a splitting field as a tower of extensions.

\begin{algo}[A splitting field of a polynomial using its Galois group]
\label{split}
Given 
\begin{itemize}
\item
a polynomial $f \in F[x]$ of degree $n$, 
where $F$ is $\Q, \Q(\alpha)$ or $\F_q(t)$, 
\end{itemize}
compute a splitting field for $f$ 
of absolute degree the order of $\Gal(f)$,
as a tower of algebraic extensions of $F$.

\begin{enumerate}
\item Compute $G = \Gal(f)$.
\item If $G$ is the trivial group then $f$ splits over $F$. Stop.
\item Compute a descending chain of subgroups $C_k$ of $G$ 
and matching invariants $I_k$ starting with $C_0 = G$ and $I_0 = 0$ by 
setting $d = \{1, \ldots, n\}$, $e = []$ and
\begin{enumerate}
\item for the stabilizer $\Stab_G(e \cup \{d_i\})$ having smallest size 
(to fix the largest field)
over all $i$ move the corresponding $d_i$ from $d$ to $e$. If this stabilizer
is not the last group in the chain $C$ then append it to $C$ and the invariant
$x_{d_i}$ to $I$,
\end{enumerate}
until $d$ is empty.
\item Compute the splitting field using Algorithm~\ref{tower} given $f, C$ and $I$.
\end{enumerate}
\end{algo}

Note that a different choice of the chain of groups will result in an isomorphic
field. The chain of groups chosen in Algorithm~\ref{split} is to construct the 
tower with the largest extensions at the bottom of the tower.

Now we discuss the details necessary to compute a resolvent polynomial
over a field in a tower which will be the defining polynomial of the next
extension in the tower. Following this we will discuss the bounds required by
this Algorithm.

\begin{algo}[A resolvent polynomial over a field in a tower]
\label{next}
Given 
\begin{itemize}
\item a polynomial $f \in F[x]$ of degree $n$, 
where $F$ is $\Q, \Q(\alpha)$ or $\F_q(t)$,
\item a tower $K$ of extensions of $F$ corresponding to a chain of subgroups of
$\Gal(f)$ of length $k-1$, 
\item a transversal $T_k = C_{k-1} // C_k$, and a $C_{k-1}$-relative $C_k$-invariant $I_k$,
\item an absolute basis $B$ of $K$, formed by taking the product of power
bases $\{\alpha_i^j\}$ of $K$ and its coefficient fields,
and a set $P$ containing all products of permutations in
all transversals for the pairs of consecutive groups in the chain already used,
\end{itemize}
compute the resolvent polynomial $g \in K[x]$ which will define the next extension in 
the tower corresponding to this transversal and invariant.
\begin{enumerate}
\item
Compute the roots $\{r_i\}_{1\le i \le n}$ of $f$ in a local splitting field
to enough precision~\cite[Sections 3.2 and 3.8.1]{suth_galois_global}
to find a Tschirnhausen transformation $T$ such that 
$I_k^{\tau}(T(r_1), \ldots, T(r_n))$ is unique for each $\tau \in T_k$.
\item\label{bound_step}
Compute the necessary bounds for recognizing the coefficients we compute with
respect to the absolute basis of $K$ as elements of $F$ and
the roots $\{r_i\}_{1\le i \le n}$ of $f$ in a local splitting field
to the necessary precision to use these bounds.
\item\label{roots}
Evaluate the invariant $I_k$ at permutations of 
$\{T(r_i)\}_{1 \le i < n}$ to gain the conjugates
$$[[c_{\tau\pi} : \tau \in T_k] : \pi \in P] = [[I_k^{\tau\pi}(T(r_1), \ldots, T(r_n)) : \tau \in T_k] : \pi \in P]$$
of the primitive element $\beta$ of the extension being constructed
with respect to the extension $K/F$.
\item\label{ele_sym} Compute all the conjugates of all the non leading coefficients $g_i$ 
of $g$ as elementary
symmetric functions $[e_{i, \pi} : 1 \le i \le \#T_k]$ in $[c_{\tau\pi} : \tau \in T_k]$ for each $\pi \in P$.
\item\label{dual} Compute the conjugates of the dual basis $d_i/h'(\alpha)$ where $h$ is the relative defining
polynomial of $K$ over the next largest coefficient field in the tower,
$d(x) = h/(x-\alpha)$ with coefficients $d_i$
and $\alpha$ is a root of $h$, a primitive element for $K$~\cite[III, Prop 2]{lang_algebra}.
This is trivial when $K = F$.
\item\label{matrix} 
For each non leading coefficient $g_i$ of $g$,
compute the coefficients with respect to the absolute basis of $K$ 
by multiplying
the $i$th vector of conjugates $[e_{i, \pi} : \pi \in P]$
by the dual basis matrix $[m_{\pi j}]_{\pi\in P, 1 \le j \le \#P}$.

\item\label{retrieve_coeffs} For each non leading coefficient $g_i$ the resulting vector of coefficients 
can be mapped into $F^{\deg{K/F}}$ using the techniques of the relevant Galois
group computation, for example~\cite[Section 3.2.1]{suth_galois_global},
and these images can be used in a linear 
combination with the absolute basis $B$ of $K$ to gain the coefficient $g_i \in K$. The leading coefficient of $g$ is 1.
\end{enumerate}
\end{algo}

Assuming our input polynomial $f$ was monic with integral coefficients, all roots
$r_i$ of $f$ are algebraic integers or integral algebraic functions. Since all our invariants are in $R[x_1, \ldots, x_n]$, where $R = \Z, \F_q[t]$, all
quantities computed in the algorithm are algebraic integers or integral algebraic functions. The algorithm
recursively constructs an absolute basis by forming products of the
basis elements of each level, thus obtaining a field basis for the 
splitting field (and the relevant subfields). However, the integral elements
we need may not have integral coefficients with respect to this basis.
In order to overcome this, we use a different basis:
For any equation order $R[x]/h$ for an integral monic irreducible $h$,
the integral closure is contained in $1/h'(x) R[x]/h$, thus
scaling the product basis by $1/h'(x)$ renders the coefficients integral.
When computing the bounds, this additional multiplication has
to be dealt with.

\subsection{Bounds}
\label{bounds}
As in~\cite[Section 3.8]{suth_galois_global} it is necessary to have bounds
in order to recognize integral elements of $F$ from the preimages mapped back
from the completion which could have been of a truncation, rather than an exact
element. These bounds are required in Step~\ref{bound_step} in order to compute 
the precision the roots of $f$ should be computed to.
In the splitting field case we are not making decisions
based on these bounds as done when computing a Galois group
but these bounds are necessary to compute a precision to
use for the computation of the roots such that the computations of the resulting
defining polynomials are accurate.

We describe the computation of these bounds here.
Over characteristic $0$ fields it is useful to consider power sums of the 
roots in order to obtain better bounds.
By the theorem on elementary symmetric functions, we can 
compute the power sums first and then use Newton relations to obtain the 
elementary symmetric functions. While bounds obtained this way still have
the same dependence on the degree, the number of terms involved is much smaller,
hence in characteristic 0 the bounds obtained this way are better (smaller).

Express the conjugates as
$$I^{\tau}\left(T(r_1), \ldots, T(r_n)\right) = \sum_j c_j \prod_i T(r_i)^{d_{ij}}$$
(in which form any invariant can be written), remembering that the invariants
may have coefficients $c_j \in \F_q[t]$ and not only in $\Z$ when the
characteristic is not $0$~\cite[Section 3.7]{suth_galois_global}. 
The defining polynomial we are calculating is
$$\prod_{\tau \in T_k} (x-I^{\tau}(T(r_1), \ldots, T(r_n)))$$ 
but we need bounds on coefficients with respect to an absolute basis.

When $f$ is over a number field its roots can be embedded in the complex field
and their complex modulus can be bounded. When $f$ is over a function
field we can instead bound the degrees of the 
roots~\cite[Section 3.8]{suth_galois_global}. Over a characteristic
0 function field we require both types of bounds on the degree and the 
coefficients as described in~\cite{hilbert}.

We will consider
these bound and precision computations for each type of arithmetic field.

\subsubsection{Number fields}
When $F$ is the rational field or a number field we can compute a bound on 
the complex modulus of the roots of $f$ and evaluate any transformations used
at this bound and
$I_k$ at these evaluations. This gives a bound $M_k$ on the $c_{\tau\pi}$.
We bound the power sums of the $c_{\tau\pi}$ by
$\#T_k M_k^{\#T_k}$ with the exponentiation covering the powers of $c_{\tau\pi}$
computed and the product with $\#T_k$ covering the sum of such powers.

Using Hadamard's inequality 
we bound the determinant of the dual basis matrix,
taking sign into account by multiplying by 2, by 
$$\lceil 2 \deg(K/F) (\deg(K/F)-1)^{(\deg(K/F)-1)/2} \prod_{j=1}^{k-1} M_j^{\prod_{i=j}^{k-1} \#T_i} / \disc(K/F)^{1/2} \rceil.$$ 

Coefficients are computed from power sums at the end of Step~\ref{retrieve_coeffs} of Algorithm~\ref{next} so mapping back from the local splitting
field applies to the power sums rather than the coefficients themselves, hence
we bound the power sums.


\subsubsection{Global Function Fields}
There are several parts to a bound for these coefficients mapped back to $K$.

\begin{description}
\item[Bounding the degrees of the elementary symmetric functions]
($\#T_k M_k$) We bound the degrees of the roots of $f$. To bound the 
evaluation of a transformation
at these roots we multiply this degree bound by the degree of the transformation
and to bound the evaluation of $I_k$ we multiply by the total degree of $I_k$. 
The details
of this bound when $c_j \in \F_q(t)$ rather
than in $\F_q$ are given in~\cite[Section 3.8]{suth_galois_global}.
This bounds $c_{\tau\pi}$ by $M_k$. The elementary symmetric 
functions in these conjugates will have degree at most $\#T_k M_k$ since 
there are $\#T_k$ many conjugates all bounded by $M_k$.

\item[Compensating for denominators ($(\#T_k - 1) M_k$)]

Since we map integral elements back from the local splitting field we need
to multiply by denominators before applying the map and divide afterwards.
These denominators are
the product of the evaluation of the derivatives of the defining polynomials
in the tower of which $K$ is the top field, at the root of the defining 
polynomial used as the primitive element of each extension in the tower.
The denominators in the local splitting fields are multiplied with the 
coefficients before multiplying by the dual basis matrix in 
Step~\ref{matrix} of Algorithm~\ref{next}.
As they are evaluations of the derivatives of polynomials of degree $\#T_k$
at conjugates they are bounded by $(\#T_k-1) M_k$.

\item[Bounding the result of the matrix multiplication] ($\#T_k M_k$)
Hadamard's inequality applies only to matrices over the complex numbers 
so cannot be used over function fields.
However, a degree bound for each
$$\sum_{\pi \in P} (\sum_{\tau \in T_k} c_{\tau \pi}^j) m_{\pi i},\ \ \  i \in [1, \ldots, \#P]$$
will be the sum of the bound for $(\sum_{\tau \in T_k} c_{\tau \pi}^j)$,
computed above and the bound for the $m_{\pi i}$.
So we now only need to bound the entries of the matrix $[m_{\pi i}]$
in Step~\ref{matrix}.

The matrix contains coefficients computed from conjugates bounded by $M_k$ and as above for the elementary symmetric functions we have a bound of $\#T_k M_k$.

\item[The final bound]
The product of an elementary symmetric function by a denominator is bounded by $\#T_k M_k + (\#T_k - 1)M_k$.
Therefore, the degrees of the entries of the product matrix resulting from Step~\ref{matrix}
are bounded by $\#T_k M_k + (\#T_k - 1)M_k + \#T_k M_k$.

\end{description}

\subsubsection{Characteristic $0$ Function Fields}
Bounds for calculations of polynomials over $\Q(t)$ will be a combination
of the bounds for number fields and characteristic $p$ function fields.
Bounds will be on degrees as for the characteristic $p$ function fields 
and on coefficients of these functions as for the number fields. Using power
sums should be the more efficient approach in this case and there will be
no problems arising from attempting to divide by multiples of $p$ when computing
elementary symmetric functions from the power sums using Newton's relations
as was a problem in characteristic $p$.

\subsection{The remaining details}
The invariants and local splitting fields of $f$
available for use with number fields are mostly described
in~\cite{fieker_klueners}, those for use with characteristic $p$ function 
fields are mostly described in~\cite{suth_galois_global} and those
for characteristic 0 function fields in~\cite{hilbert}.

\begin{theorem}
\label{th_next}
Algorithm~\ref{next} computes the resolvent polynomial 
$$g = \prod_{\tau \in T_k} (x - I_k^{\tau}(T(r_1), \ldots, T(r_n))) \in K[x]$$ 
required by Algorithm~\ref{tower} Step~\ref{resolve}.
\end{theorem}

\begin{proof}
In Step~\ref{roots} we computed 
$c_{\tau\pi} = I_k^{\tau\pi}(T(r_1), \ldots, T(r_n))$ 
in the local splitting field.
The leading coefficient of the resolvent polynomial $g$ is set to 1.
The conjugates of the non leading coefficients are computed as elementary 
symmetric functions in $c_{\tau\pi}$ for each $\pi$.

The necessary bounds have been explained above in Section~\ref{bounds}
and the roots of $f$ are computed to enough precision to use these bounds
as explained in~\cite[Section 3.8.1]{suth_galois_global} for example.



In Steps~\ref{dual} to~\ref{retrieve_coeffs} we convert from the conjugates of the coefficients
to the coefficients themselves. 
Since $\{d_j/f'(\alpha)\}_{0 \le j < n}$ is a basis dual to 
$\{\alpha^i\}_{0 \le i < n}$,
$$\mathrm{Tr}\left(\alpha^i \frac{d_j}{f(\alpha)}\right) = \begin{cases} 1, &j = i\\
0, &j \not= i \end{cases}.$$
But this trace is the product of the matrix $A$ whose columns contain 
conjugates
of powers of $\alpha$ with the matrix whose rows are the conjugates of the
elements of the dual basis so these matricies are inverse to each other.
For $\alpha \in K$, $\beta$ in the coefficient ring of $K$, $\{\gamma_k\}$
a dual basis for $\{\alpha^i\}$ and $\{\omega_l\}$ a dual basis for $\{\beta^j\}$
\begin{align*}
\mathrm{Tr}(\mathrm{Tr}(\alpha^i \beta^j \gamma_k \omega_l)) 
	&= \mathrm{Tr}(\beta^j \omega_l \mathrm{Tr}(\alpha^i \gamma_k)) \\
			   & = \begin{cases} 0, &i \not= k\\
			   \mathrm{Tr}(\beta^j \omega_l), &i = k \end{cases}\\
			   &= \begin{cases} 0, &j \not= l \text{ or } k \not= i\\
			   1, &j = l \text{ and } k = i\end{cases}
\end{align*}
so the product of the dual bases is a dual for the product basis and the 
corresponding matrices for the product basis and the product of the dual bases
are inverses.

The coefficients of $g$ will be given 
by the solution of $\mathbf{x} A = \mathbf{b}$ where $\mathbf{b}$ is the vector 
$[\sum_i x_i \alpha^i, \ldots \sum_i x_i (\alpha^i)^{(n)}]$ of conjugates.
We solve this system by multiplying by the dual basis matrix which is the 
inverse of $A$ to obtain the $x_i$ in a local splitting field. Mapping these
$x_i$ back to $F$ we then compute the coefficients $\sum_i x_i B_i$ using the 
absolute basis $B = \{B_i\}$.

\end{proof}

\begin{theorem}
\label{th_split}
Algorithm~\ref{tower} computes a splitting field of a polynomial
$f \in F[x]$ where $F = \Q, \Q(\alpha), \F_q(t)$ if the chain $C$ of 
subgroups ends with the trivial group and
Algorithm~\ref{split} computes a splitting field of such 
polynomials. 
\end{theorem}

\begin{proof}
The $k$th extension will have degree $\#T_k$ as in Theorem~\ref{th_fixed}, as
$I$ being a $C_k$-relative $C_{k+1}$-invariant guarantees the roots of $g$
are not in a subfield of $K[x]/g$,
so the field extension computed by Algorithm~\ref{tower} will have degree
$$\prod_k \#T_k = \#G/\#C_1 \times \cdots \times \#C_l/1 = \#G$$ over $F$,
the same degree as a splitting field.
Let $S_f$ be a splitting field for $f$ over $F$ so that $G$ is the Galois group
of $S_f/F$.

By Theorem~\ref{th_next} the polynomials computed by Algorithm~\ref{next} are
the resolvent polynomials needed for defining the extensions in the tower.

By the Galois correspondence,
each subgroup $C_k$ in the chain is the Galois group of $S_f$ as an extension 
of the fixed field of $C_k$.
We know that the fixed field of $C_1$ is given by such a resolvent polynomial
$g$ by Theorem~\ref{th_fixed}. Since $C_{k+1} \subset C_k$ and $C_k$ is the Galois
group over the fixed field of $C_k$
we can compute the fixed field of $C_{k+1}$ similarly as an extension of
the fixed field of $C_k$.
Finally we compute the fixed field of $\{1\}$ as an extension of the fixed
field of $C_l$ but the fixed field of the trivial group is $S_f$. Therefore,
Algorithm~\ref{tower} computes a splitting field of $f$.

Algorithm~\ref{split} returns a splitting field of degree the order of the
Galois group so this splitting field will have minimial degree.
\end{proof}

\hrule
\begin{example}
\label{eg_split}
Continuing Example~\ref{eg_fixed_split},
below the splitting field is computed as a tower of extensions over 
$K$ using Algorithm~\ref{split} above. 
We see the field defined by the input polynomial as well as the 
further extensions required to find the splitting field.
While the time taken is twice more than in Example~\ref{eg_fixed_split} it
is also one seventh of Example~\ref{eg_fact}.

\begin{verbatim}
> time GSF := GaloisSplittingField(f : Roots := false);
Time: 0.290
> Fqta<aa> := CoefficientField(GSF);
> _<y> := PolynomialRing(Fqta);
> GSF:Maximal;
  GSF
   |   y^2 + (62*t/(t^3 + 2*t^2 + 4*t + 2)*aa^5 + (2*t + 2)/
   |   (t^3 + 2*t^2 + 4*t + 2)*aa^4 + 29*t^2/(t^3 + 2*t^2
   |   + 4*t + 2)*aa^3 + (50*t^2 + 50*t)/(t^3 + 2*t^2 + 4*t
   |   + 2)*aa^2 + (8*t^3 + 4*t^2 + 8*t + 4)/(t^3 + 2*t^2 +
   |   4*t + 2)*aa + (92*t^3 + 92*t^2)/(t^3 + 2*t^2 + 4*t +
   |   2))*y + (2*t + 2)/(t^3 + 2*t^2 + 4*t + 2)*aa^5 + ....
Fqta<aa>
   |         x^6 + 98*t*x^4 + (2*t + 2)*x^3 + 3*t^2*x^2 +
   |            (6*t^2 + 6*t)*x + 100*t^3 + t^2 + 2*t + 1
Univariate rational function field over GF(101)
\end{verbatim}
\hrule
\end{example}

\section{Solution of polynomials by radicals}
\label{sect_radical}

For polynomials with solvable Galois groups
we can also construct a tower consisting of radical extensions in which the polynomial splits.

In order to check whether a polynomial can be solved by radicals, we first 
compute the Galois group and check it is solvable. 
If so, since a solvable finite group is a group with a 
composition series $C$ all of whose quotients are cyclic groups of prime order,
we use this series $C$ of subgroups in Algorithm~\ref{tower}
to compute a tower 
of cyclic extensions which we can convert to radical extensions. This 
conversion requires handling the necessary roots of unity. 
We now detail the steps of this algorithm.

\begin{algo}[A tower of radical extensions in which a polynomial splits]
\label{radical_algo}
Given a polynomial $f$ of degree $n$ over $F = \Q, \Q(\alpha)$ or $\F_q(t)$ 
compute a tower of radical extensions of $F$, 
of minimal degree 
over which $f$ splits using the Galois group of $f$.
\begin{enumerate}
\item Choose a prime good for the computation of the Galois group of
the product of $f$ and the cyclotomic polynomials for $m$th roots of unity
for $m < n$.
\item Compute $\Gal(f)$ and check it is solvable. If not return.
\item\label{unity} Determine which roots of unity are needed and compute the Galois group $G$
of the LCM of $f$ and the associated cyclotomic polynomials.
\item Compute a chain $C$ of subgroups, starting with $G$, then those which
stabilize an increasing number of roots of unity and ending with 
the composition series of the current last subgroup in the chain $C$.
\item\label{step_tower} Using Algorithm~\ref{tower}, compute the tower of 
cyclic fields from the LCM of
$f$ and the cyclotomic polynomials, $C$ and invariants either computed 
using~\cite{fieker_klueners,suth_galois_global} or indeterminates corresponding 
to the roots of unity.

\item\label{cyc_rad_step} Transform cyclic extensions to radical extensions using Algorithm~\ref{cyclic_radical}.
\end{enumerate}

\end{algo}

Note that Step~\ref{unity} requires the computation of a Galois group of a reducible
polynomial which has been discussed in~\cite[Section 4]{suth_galois_global}.
The computation of a good prime for this computation of a Galois group is also
described in that section.
It is for this reducible polynomial that the tower of radical extensions will be
a splitting field.

To determine which roots of unity are necessary in Step~\ref{unity} 
we consider the $p$-th roots of unity
for those primes $p$ dividing the order of $G$. We also consider the $q$th roots
of unity for those primes $q$ dividing the $p-1$ and the $r$th roots of unity for
those primes $r$ dividing the $q-1$ and so on. We compute GCDs of the cyclotomic
polynomials for these roots of unity with $f$ 
and so remove any roots of unity from the cyclotomic polynomials which are 
already roots of $f$.

For function fields in any 
characteristic the cyclotomic polynomials which are used to handle adjoining
necessary roots of unity define constant field extensions so instead of using
the extension $\F_q(t)(\alpha)$ where $\alpha$ is a root of a polynomial
of degree $r$ over $\F_q$ we can instead use the field $\F_{q^r}(t)\cong \F_q(t)(\alpha)$.
%
Making such extensions of the constant field instead of the function field
makes the function
field extensions smaller and improves the efficiency of Algorithm~\ref{radical_algo} and any use of the field afterwards.

It is also possible for global function fields
that the cyclotomic polynomials are reducible
over the constant field as the constant field already contained some of those roots
of unity, unlike the rational field which only contains 2nd roots of unity.
Roots of unity which are already included in the constant field do not need 
to be adjoined.

Note that we could just use the composition series of $G$ to compute a tower of 
radical extensions. However, by computing subgroups which stabilise the roots
of unity we put these extensions at the base of the tower and for function
fields this simplifies the construction of these constant field extensions.

We now state the algorithm we use for Step~\ref{cyc_rad_step}.

\begin{algo}[Convert a cyclic extension to a radical extension]
\label{cyclic_radical}
Given a cyclic extension $K'/K$ of prime degree $d$ compute an isomorphic radical extension.
\begin{enumerate}
\item Compute a generator $\sigma$ for the automorphism group of $K'/K$.
\item \begin{description}
\item[if $d$ is equal to the characteristic $p$ of $K$]
Return $$K[x]/(x^d - x - (b^p-b))$$ where
$$b = 1/\mathrm{Tr}(\theta)\sum_{i=1}^{p-1} i \sigma^i(\theta),
\exists \theta \in K', \mathrm{Tr}(\theta) \not= 0.$$

\item[if $d = 2$] Let $K'/K$ be defined by $x^2 + a_1 x + a_0$.
Return $$K[x]/(x^2+a_0-a_1^2/4).$$

\item[if $d > 2$]
Return $K[x]/(x^d - b^d)$ where $b = \sum_i \zeta^i \sigma^{(n-i)}(a) \in K'$,
$a\in K'$ such that $b\not= 0$ and $\zeta$ is a $d$th root of unity.
\end{description}
\end{enumerate}
\end{algo}

There are algorithms available for computing automorphisms of number 
fields~\cite{AcKl1, Kl2} and function fields~\cite{antshessautiso}. While computing automorphisms of
large towers of extensions of number fields appears to be efficient for 
the extensions $K'/K$ 
encountered in computing a tower of radical extensions, the 
computation of automorphisms of large towers of radical 
function field extensions appeared prohibitive.
From the Composition series and the Galois correspondence 
the Galois group of $K'/K$ is a cyclic group of order $d$ which is prime.
Such groups can be generated by each
%
%
of their non identity elements so any
non identity automorphism can be used as the generating automorphism $\sigma$
in the calculation of a radical generator.
The cost of this though is the computation of a root of the defining polynomial
of $K'$ in $K'$ as we need an image under the chosen generating automorphism $\sigma$.

\begin{theorem}
Algorithm~\ref{radical_algo} computes a tower of radical extensions of minimal 
degree over which the polynomial $f$ splits.
\end{theorem}

\begin{proof}
From Theorem~\ref{th_split} we have that Algorithm~\ref{radical_algo} 
Step~\ref{step_tower} computes
a splitting field for $f \prod_j c_{p_j}$ where $c_{p_j}$ are cyclotomic 
polynomials whose roots are the primitive $p_j$th roots of unity. 

If a prime $q$ divides the order of $\Gal(f)$ then there will be an extension
defined by $x^q - a$ for some $a$ appearing in a tower of radical extensions
over which $f$ splits. The roots of this defining polynomial in the splitting
field will be $\{ \zeta_q^i a^{1/q} \}_{0 \le i < q}$ 
and so we need to include $\zeta_q$
in the radical tower to ensure this tower is a splitting field.
However, the minimal polynomial of $\zeta_q$ is a cyclotomic polynomial of
degree $\phi(q) = q-1$ so we also will have extensions defined by $x^r - b_r$
in the radical tower for $r \mid q - 1$ and some $b_r$. Hence we must 
include such $\zeta_r$ and so on.
For primes which do not divide $\#\Gal(f)$ or any such $q-1$, there will be no
extension of this prime degree in the radical tower and such roots of unity
are not necessary.

In order to ensure that all quotients of consecutive groups in the chain
$C$ are of prime order, we contruct a composition series of $\Gal(f)$
which contains $S$, the stabilizer for a subfield generated by roots
of unity in the splitting field. Since $\Gal(f)$ is solvable and 
$S \subset \Gal(f)$, $S$ is solvable~\cite[Theorem 13.2(1)]{stewart_galois}.   
 
Since the subfield is abelian, it is normal and so is its stabilizer,
therefore, $\Gal(f)/S$ is solvable~\cite[Theorem
13.2(2)]{stewart_galois}. Since $S$ and $\Gal(f)/S$ are solvable,
$\Gal(f)$ is solvable~\cite[Theorem 13.2(3)]{stewart_galois} 
(as already known) and further 
there is a composition series for $\Gal(f)$ containing $S$.
Therefore all such stabilizers can be
included in a composition series for $\Gal(f)$.

Since the chain of subgroups passed in to Algorithm~\ref{tower} consists of
groups whose quotients with consecutive groups are cyclic of prime order from 
the composition series and as above from the roots of unity, 
the field extensions in the tower will be cyclic of prime degree.

To see that Algorithm~\ref{cyclic_radical} returns a radical extension 
isomorphic to a given cyclic extension,
\begin{description}
\item[if $d$ is equal to the characteristic $p$ of $K$]
For degree $p$ extensions in characteristic $p$, 
a wider definition of solvability by radicals can be 
used~\cite[p 129]{stewart_galois}.
Essentially instead of computing a Kummer generator we compute 
an Artin--Schreier generator,
a root of an Artin--Schreier polynomial~\cite[Remark p 147]{stewart_galois}. 
Artin--Schreier generators exist in the same way Kummer generators 
exist~\cite{stichtenoth}.
In characteristic $p$ we have 
\begin{itemize}
\item $x^p-a$ is inseparable. It only has one distinct root, not $p$ roots, and
its derivative is zero.
\item In a cyclic degree $p$ extension $K'/K$ there exists $\beta$ such that
$\beta^p-\beta \in K$~\cite[A13]{stichtenoth}.
In this case $x^p - x - a = \prod_{i=1}^{p-1} (x - (\beta + i))$ defines an Artin--Schreier extension.
\end{itemize}
From~\cite[Theorems 6.3 and 6.4]{lang_algebra},
$\beta = 1/\mathrm{Tr}(\theta)\sum_{i=1}^{p-1} i \sigma^i(\theta)$,
$\exists\theta \in K', \mathrm{Tr}(\theta) \not= 0$
is a primitive element such that $b^p - b \in K$.

\item[if $d = 2$] 
Completing the square reveals that $\alpha + a_1/2$
is a zero of $x^2+a_0-a_1^2/4$. 

\item[if $d > 2$]
From~\cite{waerden_modern}, $b = \sum_i \zeta^i \sigma^{(n-i)}(a) \in K'$,
where $\zeta$ is a $d$th root of unity,
is a primitive element such that $b^d \in K$.
\end{description}

\end{proof}

\hrule
\begin{example}
Below is a degree 6 polynomial which can be solved by radicals, many cannot be. 
Computing the {\tt GaloisSplittingField} of this polynomial as shown in
Example~\ref{eg_split} above results in a degree 2
extension of the degree 6 extension defined by this polynomial. Solving by
radicals results in this splitting field being expressed as 3 extensions, 
the lowest one being a constant field extension hidden in $\F_{101^2}$.
The difference comes from the chain of subgroups used --- in this case using 
the composition series
we used more subgroups and so there are more extensions.

\begin{verbatim}
> Fqt<t> := FunctionField(GF(101));
> P<x> := PolynomialRing(Fqt);     
> f := x^6 + 98*t*x^4 + (2*t + 2)*x^3 + 3*t^2*x^2 +
>           (6*t^2 + 6*t)*x + 100*t^3 + t^2 + 2*t + 1;
> S := SolveByRadicals(f);                                   
> CS<cs> := CoefficientRing(S);
> _<t> := CoefficientRing(CS);
> S:Maximal;
    S
    |
    | $.1^2 + 100*t
    |
  CS<cs>
    |
    | $.1^3 + 8*t + 8
    |
Univariate rational function field over GF(101^2)
Variables: t
> DefiningPolynomial(ConstantField(S));
t^2 + 26
> _<w> := ConstantField(S);
> Roots(Polynomial(S, f));
[
    <S.1 + (51*w + 25)*cs, 1>,
    <100*S.1 + (51*w + 25)*cs, 1>,
    <S.1 + (50*w + 25)*cs, 1>,
    <100*S.1 + (50*w + 25)*cs, 1>,
    <S.1 + 51*cs, 1>,
    <100*S.1 + 51*cs, 1>
]
\end{verbatim}
\hrule
\end{example}
\begin{example}
{An example with a degree characteristic extension}
\begin{verbatim}
> F<t> := FunctionField(GF(5));
> P<x> := PolynomialRing(F);   
> f := x^5 - x + t;            
> GaloisGroup(f);           
Permutation group acting on a set of cardinality 5
Order = 5
    (1, 2, 5, 4, 3)
> IsSoluble($1);
true
> SolveByRadicals(f);
Algebraic function field defined over Univariate rational 
function field over GF(5) by x^5 + 4*x + t
\end{verbatim}
\hrule
\end{example}
\begin{example}
An interesting example with a degree characteristic extension.
\begin{verbatim}
> f := x^5 + x^4 + t; 
> G := GaloisGroup(f); 
> TransitiveGroupDescription(G); 
F(5) = 5:4                       
> IsSoluble(G);
true
> S := SolveByRadicals(f); 
> CS<cs> := CoefficientRing(S);
> CCS<ccs> := CoefficientRing(CS); 
> S:Maximal;
     S
     |
     |                      $.1^5 + 4*$.1 + 2/t^2*ccs*cs
   CS<cs>
     |
     |                                     $.1^2 + 2*ccs
  CCS<ccs>
     |
     |                                       x^2 + 4*t^3
Univariate rational function field over GF(5)
Variables: t
> Roots(f, S);
[   <S.1^4 + S.1^3 + S.1^2 + S.1, 1>,
    <S.1^4 + 3*S.1^3 + 4*S.1^2 + 2*S.1, 1>,
    <S.1^4 + 2*S.1^3 + 4*S.1^2 + 3*S.1, 1>,
    <S.1^4 + 4*S.1^3 + S.1^2 + 4*S.1, 1>,
    <S.1^4 + 4, 1>
]
> Roots(f, FunctionField(f));
[
    <$.1, 1>
]
\end{verbatim}
From the roots we can see that $f$ does not split over $\F_5(t)[x]/f$.
\hrule
\end{example}

\section{Geometric Galois groups}
\label{sect_geom}
Geometric Galois groups have connections
to inverse Galois theory and are an alternate approach to computing absolute 
factorizations.
They have also already been used computationally in solving a Japanese
geometry problem~\cite{morikawa}.

We present below an algorithm we have developed to compute geometric
Galois groups of polynomials $f\in \Q(t)[x]$.
Further we consider the use
of this algorithm to compute the absolute factorization of polynomials 
$f\in \Q(t)[x]$. 
This algorithm has been available in {\sc Magma}~\cite{magma223} since V2.24.

First we make the necessary definitions.

\begin{definition}
The {\em geometric Galois group}, $\GGal(f)$, of a polynomial $f \in \Q(t)[x]$ is
the Galois group of $f$ considered as a polynomial over $\C(t)$, $\Gal(f/\C(t))$. 
\end{definition}

\begin{definition}
Given a function field 
$F/k(t)$, the algebraic closure of $k$ in $F$, 
$K = \{z : z \in F \ |\ \text{$z$ is algebraic over $k$}\}$,
is the {\em full} or {\em exact constant field} of $F$.
\end{definition}

\subsection{Connections to Inverse Galois Theory}
Computations of geometric Galois groups have links to Inverse Galois Theory.
We summarise below the relevant information from~\cite[Chapter I]{malle_matzat} to which the 
references below refer.

\begin{itemize}
\item
Every finite group occurs as a Galois group over $\C(t)$ (Corollary 1.5)
and as a geometric Galois group over $\R(t)$ (Corollary 1.7) and 
$\bar \Q(t)$ (Corollary 2.3) so that the inverse Galois problem is solved over all
rational function fields $\bar k(t)$ of characteristic 0.

\item
A finite extension $\Gamma/k(t)$ is {\em geometric} if $k$ is algebraically closed in 
$\Gamma$ (Section 1).
This occurs when $k$ is the full constant field of $\Gamma$.

\item
If $H$ is a Galois group $H = \Gal(\Gamma/k(t))$ then this realization is a 
{\em $G$-realization of $H$ over $k$} if
$\Gamma/k(t)$ is a geometric Galois extension with Galois group $H$ over a rational function
field $k(t)$ (Section 5.1).
\end{itemize}

Hence any $\Gamma/\C(t)$ is a geometric extension and the geometric Galois group
$\Gal(\Gamma/\C(t))$ is a $G$-realization over $\C$. 
Since every finite abelian group has a $G$-realization over $\Q$ (Theorem 5.1),
there are geometric extensions $\Gamma/\Q(t)$ whose geometric Galois
groups are equal to $\Gal(\Gamma/\Q(t))$.


\subsection{Previous Work}
Numerous polynomials whose geometric Galois groups are equal to their
Galois groups are given in~\cite{malle_matzat}. This collection was
provided to us in {\sc Magma} readable form by J\"urgen Kl\"uners and has since been
since extended by Kl\"uners to include polynomials whose geometric Galois
groups differ from their Galois groups.

\subsection{Algorithm}
Since $\Q$ is a subfield of $\C$, $\GGal(f)$
will be a subgroup of $\Gal(f/\Q(t))$
--- as the field $\C$ is larger than $\Q$ the group is smaller.
Therefore again start with computing a Galois group $G$ --- this time over $\Q(t)$.
Using Hilbert's Irreducibility Theorem we can find $t_1, t_2$ at which the 
specializations of the polynomial have Galois
group isomorphic to $G$. Then compute the Galois group over $\Q$ of the
product of the specializations, a subgroup of the direct product of
$\Gal(f/\Q(t))$ with itself. 

Using inexact fields such as $\C$ often leads to 
problems with precision, so we avoid computing in $\C$ altogether and
instead find the largest algebraic extension
of $\Q$ which contains the algebraic numbers we require : the algebraic closure
$K$ of $\Q$ in the splitting field $S_f/\Q(t)$ which is the exact constant field
of $S_f$. 
All the constants of $S_f/\Q(t)$ will be in $K$ since $K = S_f \cap \C$.
%
We compute the geometric Galois group of $f$
as the Galois group of $f$ considered as a polynomial over $K(t)$,
although computations over $K(t)$ are also unnecessary and $K$ is not known
before hand.
The task then is to determine $K$, not knowing $S_f$ and which normal subgroup $X$ of $\Gal(f)$ fixes $K(t)$.
We compute $K$ at the same time as the geometric Galois group.

There are some bounds on the index of the geometric Galois group in $\Gal(f/\Q(t))$.
Also, since the degree of the fixed field of the geometric Galois group, $K(t)$,
is divisible by the degree of the exact constant field of $\Q(t)[x]/f$,
the order of the geometric Galois group must divide the quotient of $\#G$ by 
the degree of the exact constant field.

We need only consider the normal subgroups which satisfy these constraints, 
compute their fixed fields and look at which of these 
is a rational function field over an extension of the constant field. This
gives subgroups of $G$ which contain
the geometric Galois group and the coefficient rings of 
their fixed fields which are contained in 
the algebraic closure of $\Q$ in the splitting field of $f$ over $\Q(t)$.

We summarise the algorithm as follows and then make some notes : 

\begin{algo}[Compute the geometric Galois group of $f\in \Q(t)\lbrack x\rbrack$]\hfill\par\noindent
Given a polynomial $f \in \Q(t)[x]$ compute $\GGal(f)$
as well as the algebraic closure $K$ of $\Q$ in $S_f$.
\label{algo_ggg}
\begin{enumerate}
\item Compute $G = \Gal(f)$.
\item Specialise $t$ at small integer values. Choose $t = t_i, i = 1, 2$ 
such that $\Gal(f(t_i, x)) = G$.

\item Compute $H = \Gal(f(t_1, x) f(t_2, x))$. 

\item For normal subgroups $X$ of $G$ having index less than $[G\times G : H]$ 
and index divisible by 
$c$ where $c$ is the degree of the full constant field of $\Q(t)[x]/f$, 
\begin{enumerate}
\item[(a)] 
Compute the defining polynomial of the field $K'$ fixed by $X$ using 
Algorithm~\ref{fixed}. 

\item[(b)]
Check whether this is a polynomial over $\Q$ or whether this
defines a constant field extension.
If so $X$ contains $\GGal(f)$ and $K \supseteq K'$.
\end{enumerate}

\item The subgroup $X$ containing $\GGal(f)$
with the largest index in $G$ and smallest order
corresponds to the largest constant field extension in $S_f$ which is the
algebraic closure $K$ of $\Q$ in $S_f$,
and is $\GGal(f)$.
\end{enumerate}
\end{algo}

Let $X = \Gal(S_f/K(t))$.
The fields and groups we are considering are shown in the following diagram : 

\input{geom_gal_grp_diag}

\bigskip\noindent
where $K_i$ is the residue field of $S_{f}$ at $t-t_i$, 
defined by $f(t_i, x)$ since $0 = f(t, \alpha) \equiv f(t_i, \alpha)$ where
$\alpha$ is a root of $f(t, x)$,
and $K \subseteq K_i$ since all constants have valuation $0$ at all primes 
and so are in all residue fields. 

We can narrow down the possible subgroups to consider by
deriving constraints on the index of $\GGal(f)$ in $\Gal(f)$.
%
Using Hilbert's Irreducibility Theorem~\cite{hilbert_lang} 
we have
$$G =\Gal(f/\Q(t)) = \Gal(S_f) = \Gal(f(t_i, x)/\Q) = \Gal(S_{f_i})$$
for infinitely many $t_i\in \Q$. 
Choose $t_1, t_2$ 
such that $G = \Gal(f(t_1, x)) = \Gal(f(t_2, x))$.
%
Let 
$H = \Gal(f(t_1, x) f(t_2, x))$, the Galois group of a
reducible polynomial over $\Q$. As computed by~\cite[Section 4]{suth_galois_global}
$H\subseteq G\times G$ and since 
$\#(G\times G) = \deg(S_{f_1})\deg(S_{f_2}) = \deg(S_f)^2$ and
$\#H = \deg(S_{f_1 f_2})$ and the difference in these degrees is determined by
the intersection of $S_{f_1}$ and $S_{f_2}$,
$[G \times G : H] = [{S_f}_1 \cap {S_f}_2 : \Q]$. As $K \subseteq {S_f}_1 \cap {S_f}_2$, $$[G \times G : H] \ge [K : \Q] = [G : X].$$
This is our first bound on the index of $X$ in $\Gal(f)$.


Since $K$ will contain all the constants of 
$\Q(t)[x]/f \subseteq {S_f}$ the degree of the exact constant field of $\Q(t)[x]/f$ as an extension of $\Q$
must divide $[K : \Q]$. This narrows down the number of subgroups which are candidates
for $X$ to those whose index is at most our bound and divisible by
this field degree. 

Lastly we know that the fixed field $K(t)$ of the subgroup $X \subseteq G$
is contained in $\C(t)$ so we can discount any subgroups whose fixed fields are not isomorphic to
rational function fields over algebraic extensions of $\Q$.

\begin{theorem}
Algorithm~\ref{algo_ggg} computes the geometric Galois group of $f \in \Q(\alpha)(t)[x]$.
\end{theorem}

\begin{proof} 
Algorithm~\ref{algo_ggg} computes the Galois group $X$ of the largest constant 
field extension $K(t) = \Q(\alpha)(t) \subseteq S_f$. We have discussed above 
that $c \mid [K : \Q] = [G : X] \le [G\times G : H]$ so that these restrictions
only improve efficiency and do not remove this answer from consideration.

Since $K(t) \subset \C(t)$ we know that $\GGal(f) = \Gal(f/\C(t)) \subseteq \Gal(f/K(t))$. 
Let $g \in \Gal(f/K(t))$ so that $g$ is an automorphism which permutes the 
roots of $f$. If $g$ fixes $\C(t)$ then $g \in \Gal(f/\C(t))$.
Let $y \in \C(t) \backslash K(t)$, then 
$y \not\in S_f/\Q(t)$. But all roots of $f$ are in $S_f/\Q(t)$ so $y$ cannot be
written as an algebraic combination of roots of $f$, therefore $g(y) = y$
and $g \in \GGal(f)$.

\end{proof}

\subsection{A Non-Trivial Example}
Most geometric Galois groups are equal to the Galois group of the polynomial
but we show here an example where that is not the case.

We identify transitive groups in the form $dTn$, the $n$th transitive group of
degree $d \le 32$ according to the ordering in the database of transitive groups
in {\sc Magma}~\cite{magma217}. This is the same numbering used in
GAP~\cite{gap} when $d < 32$ where the groups have been either confirmed or
provided by Hulpke~\cite{hulpke_class}. The transitive groups of
degree 32 were provided by~\cite{cannon_holt_class}.

\medskip
\hrule
\begin{example}
\label{eg_ggg}
Let 
\begin{multline*}
f = x^9 - 3 x^7 + (-6 t + 6) x^6 + 3 x^5 + (12 t + 6) x^4 + (12 t^2 + 84 t + 11) x^3 + \\
(-6 t + 6) x^2 + (-12 t^2 - 12 t + 24) x - 8 t^3 + 24 t^2 - 24 t + 6 \in \Q(t)[x]
\end{multline*}
be a polynomial over $\Q(t)$. This polynomial was computed as a defining
polynomial of 
$(\Q(t)[x]/\langle x^3 - 2\rangle)[y]/\langle y^3-2 t-y\rangle$ over $\Q(t)$.

This geometric Galois group and the algebraic closure of the constant field in
the splitting field can be computed in {\sc Magma} by
\begin{verbatim}
> F<t> := FunctionField(Rationals());
> P<x> := PolynomialRing(F);
> f := x^9 - 3*x^7 + (-6*t + 6)*x^6 + 3*x^5 + (12*t + 6)*x^4 + 
> (12*t^2 + 84*t + 11)*x^3 + (-6*t + 6)*x^2 + 
> (-12*t^2 - 12*t + 24)*x - 8*t^3 + 24*t^2 - 24*t + 6;
> time GeometricGaloisGroup(f);
Permutation group acting on a set of cardinality 9
Order = 6 = 2 * 3
    (2, 4)(3, 5)(7, 8)
    (1, 5, 3)(2, 9, 4)(6, 8, 7)
x^6 + 78732
GaloisData of type Q(t)

Time: 1.160
\end{verbatim}
The steps taken to compute this geometric Galois group are detailed below.
\begin{enumerate}
\item 
We compute the Galois group over $\Q(t), \Gal(f)$ as $9T8$, equivalently, $S_3 \times S_3$ of order 36. 
\begin{verbatim}
> G := GaloisGroup(f); G;
Permutation group G acting on a set of cardinality 9
Order = 36 = 2^2 * 3^2
    (2, 7)(4, 8)(6, 9)
    (2, 4)(3, 5)(7, 8)
    (1, 2)(3, 9)(4, 5)(6, 7)
> TransitiveGroupIdentification(G);
8 9
> TransitiveGroupDescription(GaloisGroup(f));
S(3)[x]S(3)=E(9):D_4
\end{verbatim}
\item Specialising $f$ at $t = 1, 2$ gives 
\begin{align*}
f_1 &= x^9 - 3 x^7 - 12 x^6 + 3 x^5 + 6 x^4 - 61 x^3 - 12 x^2 + 24 x - 62,\\
f_2 &= x^9 - 3 x^7 - 18 x^6 + 3 x^5 + 18 x^4 - 109 x^3 - 18 x^2 - 214
\end{align*}
with $\Gal(f_1), \Gal(f_2)$ conjugate to $9T8$. 
\item $H = \Gal(f_1 f_2)$ is an intransitive group of order 216. 

\begin{verbatim}
> f1 := Polynomial([Evaluate(x, 1) : x in Coefficients(f)]);
> f2 := Polynomial([Evaluate(x, 2) : x in Coefficients(f)]);
> H := GaloisGroup(f1*f2); H;  
Permutation group H acting on a set of cardinality 18
Order = 216 = 2^3 * 3^3
    (11, 15)(12, 17)(13, 14)
    (1, 7, 6)(2, 5, 4)(3, 9, 8)(10, 16, 18)(11, 13, 17)
                                            (12, 15, 14)
    (2, 9)(3, 4)(5, 8)
    (1, 5, 8)(2, 9, 6)(3, 7, 4)
    (2, 4)(3, 9)(6, 7)(12, 14)(13, 17)(16, 18)
    (10, 11, 15)(12, 18, 17)(13, 14, 16)
\end{verbatim}
\item 
We compute a bound using the order 36 of $G$ and
the order 216 of $H$ to give us an index bound for $\GGal(f)$ of
$36 \times 36 / 216 = 6$.
We also know that the exact constant field of the splitting field of $f$
must contain the exact constant field of $\Q(t)[x]/f$ which has degree 3 so
the order of $\GGal(f)$ must divide $36/3=12$. 
\begin{verbatim}
> #DirectProduct(G, G)/#H;      
6
> ExactConstantField(ext<F | f>);
Number Field with defining polynomial x^3 + 
    847813960/821097*x^2 + 
    655877533925420992/2022600850227*x + 
    357337502918317868824361728/14946763412869551171 
    over the Rational Field
\end{verbatim}
\begin{enumerate}
\item There are 2 normal subgroups of $\Gal(f)$, both isomorphic to $S_3$ of order 6, which satisfy this index and order restriction.
\begin{verbatim}
> subgroups := NormalSubgroups(G : IndexLimit := 
>              (36*36) div 216, OrderDividing := 
>              #G div 3);
> [IsIsomorphic(Sym(3), x`subgroup) : x in subgroups];              
[ true, true ]
\end{verbatim}
\item The corresponding fixed fields are defined by $x^6 + 78732$ and
$x^6 - 54 x^4 + 729 x^2 + 78732 t^2 - 2916$.
\begin{verbatim}
> GaloisSubgroup(f, subgroups[1]`subgroup);
$.1^6 - 54*$.1^4 + 729*$.1^2 + 78732*x^2 - 2916
((x2 + (x7 + x3)) + (2 * ((x9 + x6) + x1)))
> GaloisSubgroup(f, subgroups[2]`subgroup);
x^6 + 78732
((2 * (x1 + (x3 + x5))) + (x2 + (x9 + x4)))
\end{verbatim}
\end{enumerate}
\end{enumerate}

Only one of the fixed fields has a defining
polynomial over $\Q$ : $x^6 + 78732$. 
The other fixed field, 
which is defined by $x^6 - 54 x^4 + 729 x^2 + 78732 t^2 - 2916$,
is not isomorphic to its exact constant field, which is $\Q$.
Therefore the first fixed field is
the exact constant field of the splitting field of $f$ and the corresponding
group, isomorphic to $S_3$, is the geometric Galois group of $f$. 

\begin{verbatim}
> ExactConstantField(ext<F | $2>);
Rational Field
Mapping from: FldRat: Q to Algebraic function field defined over
Univariate rational function field over Rational Field by
x^6 - 54*x^4 + 729*x^2 + 78732*t^2 - 2916
\end{verbatim}
\end{example}
\hrule

\subsection{Absolute Factorization}
\label{sect_abs}
We can use the information gained about the full constant field of the 
splitting field using Algorithm~\ref{algo_ggg} to compute an absolute factorization.
Thus we can compute the absolute factorization of $f\in\Q(t)[x]$ without using
calculations over $\C$ by factoring $f \in K(t)[x]$.
This approach is different to those described in~\cite{Cheze2005}.

\begin{definition}
The {\em absolute factorization} of a polynomial $f \in \Q(t)[x]$ is 
the factorization of $f$ as a polynomial over $\C(t)$.
\end{definition}

\medskip
\hrule
\begin{example}
Continuing Example~\ref{eg_ggg} we compute the absolute factorization of 
\begin{multline*}
f = x^9 - 3 x^7 + (-6 t - 6) x^6 + 3 x^5 + (12 t - 6) x^4 + (12 t^2 - 84 t + 11) x^3\\
+ (-6 t - 6) x^2 + (-12 t^2 + 12 t + 24) x - 8 t^3 - 24 t^2 - 24 t - 6 \in \Q(t)[x],
\end{multline*}
by computing 
a factorization of $f$ over 
$\Q(t)[\alpha] = \Q(t)[x]/\langle x^6 + 78732\rangle$,
the exact constant field of the splitting
field of $f$ computed whilst calculating the geometric Galois group of $f$. 
There are 3 cubic factors of $f$ over $\Q(t)[\alpha]$ : 
\begin{align*}
y^3 - 1/486 \alpha^4 y^2 + &(-1/9 \alpha^2 - 1) y + 1/1458 \alpha^4 - 
2 t + 2 \\
y^3 + (-1/972 \alpha^4 - 1/2 \alpha) y^2 + &(1/2916 \alpha^5 + 1/18 \alpha^2 - 1) y 
+ 1/2916 \alpha^4\\ &+ 1/6 \alpha - 2 t + 2 \\
y^3 + (-1/972 \alpha^4 + 1/2 \alpha) y^2 + &(-1/2916 \alpha^5 + 1/18 \alpha^2 - 1) y
+ 1/2916 \alpha^4\\ &- 1/6 \alpha - 2 t + 2
\end{align*}

\begin{verbatim}
> K := ext<F | x^6 + 78732>;
> _<y> := PolynomialRing(K);
> Factorization(Polynomial(K, f));
[
    <y^3 + (-1/972*K.1^4 - 1/2*K.1)*y^2 + 
        (1/2916*K.1^5 + 1/18*K.1^2 - 1)*y + 
        1/2916*K.1^4 + 1/6*K.1 - 2*t + 2, 1>,
    <y^3 + (-1/972*K.1^4 + 1/2*K.1)*y^2 + 
        (-1/2916*K.1^5 + 1/18*K.1^2 - 1)*y + 
        1/2916*K.1^4 - 1/6*K.1 - 2*t + 2, 1>,
    <y^3 + 1/486*K.1^4*y^2 + (-1/9*K.1^2 - 1)*y - 
        1/1458*K.1^4 - 2*t + 2, 1>
]
\end{verbatim}
In this case, we can verify as follows 
that splitting these factors results in no
further constant field extension and thus this factorization is the absolute
factorization of $f$.

The extension of $\Q(t)[\alpha]$ by the last factor above
is a field $F_{18}$
of degree 18 over $\Q(t)$ whose exact constant field is
of degree 6 over $\Q$; hence $F_{18} \cap \C = \Q[\alpha]$. Over $F_{18}$ 
$f$ factors into linear and quadratic factors.
Extending $F_{18}$ by one of the quadratics gives a degree 36 field which
must be a splitting field of $f$ as the degree matches the order of 
$\Gal(f)$. 

Therefore the
factorization of $f$ into the three cubics above is the absolute factorization of $f$,
as the other field extensions necessary
to factorize $f$ further, $F_{18}$ and $S_f$,  
are not contained in $\C(t)$, which 
$\Q(t)[\alpha] = \Q(t)[x]/\langle x^6+78732\rangle \cong (\Q[x]/\langle x^6+78732\rangle)(t)$ is.

The work in the last paragraph was a verification for illustrative 
purposes only and is not necessary for the computation of the absolute factorization.

These calculations were performed after computing the geometric Galois group in a little over 1s, in 
0.08s for the absolute factorization and 25 min for the further factorizations,
construction of extensions and mostly for computations of algebraic closures of $\Q$.

\end{example} 
\hrule

\section*{Acknowledgements}
Examples have been run using {\sc Magma} V2.25-6~\cite{magma225}
on a Intel\regsym{} Core\trademark{} i7-3770 CPU @ 3.40GHz with 32GB RAM.
Some sections of this paper were most recently discussed during a visit the 
second author made to Kaiserslautern in 2019 following WIN-E3 while other
sections were discussed when the second author visited Kaiserslautern for
ANTS XII in 2016. The second author would like to thank A.\! Prof.\!\! Donald Taylor for some
productive discussions regarding some of the necessary group theory in the 
paper and Dr Mark Watkins for discussions regarding dual bases.

\bibliographystyle{amsalpha}
\bibliography{thesis}

\end{document}

%% file: geom_gal_grp_diag
\setlength{\unitlength}{3947sp}%
\begingroup\makeatletter\ifx\SetFigFont\undefined%
\gdef\SetFigFont#1#2#3#4#5{%
  \reset@font\fontsize{#1}{#2pt}%
  \fontfamily{#3}\fontseries{#4}\fontshape{#5}%
  \selectfont}%
\fi\endgroup%
\begin{picture}(3180,2000)(2311,-2000)
\thinlines
{\put(2600,-511){\line( 0,-1){600}}
}%
{\put(2600,-1411){\line( 0,-1){400}}
}%
{\put(3200,-300){\line( -1,0){200}}
}%
{\put(3200,-300){\line( 0,-1){1000}}
}%
{\put(3200,-1300){\line( -1,0){200}}
}%
{\put(4600,-511){\line( 0,-1){600}}
}%
{\put(4600,-1411){\line( 0,-1){100}}
}%
{\put(4600,-1761){\line( 0,-1){100}}
}%
{\put(5700,-1950){\line( -1,0){200}}
}%
{\put(5700,-1650){\line( 0,-1){300}}
}%
{\put(5700,-1650){\line( -1,0){60}}
}%
{\put(5450,-300){\line( -1,0){300}}
}%
{\put(5450,-300){\line( 0,-1){1000}}
}%
{\put(5450,-1300){\line( -1,0){300}}
}%
{\put(6685,-511){\line( 0,-1){300}}
}%
{\put(6685,-1161){\line( 0,-1){650}}
}%
{\put(7500,-300){\line( 0,-1){1650}}
}%
{\put(7500,-300){\line( -1,0){300}}
}%
{\put(7500,-1950){\line( -1,0){300}}
}%
\put(2300,-436){\makebox(0,0)[lb]{\smash{{\SetFigFont{12}{14.4}{\rmdefault}{\mddefault}{\updefault}{$S_f/\C(t)$}%
}}}}
\put(2505,-1336){\makebox(0,0)[lb]{\smash{{\SetFigFont{12}{14.4}{\rmdefault}{\mddefault}{\updefault}{$\C(t)$}%
}}}}
\put(2505,-2036){\makebox(0,0)[lb]{\smash{{\SetFigFont{12}{14.4}{\rmdefault}{\mddefault}{\updefault}{$\Q(t)$}%
}}}}
\put(4560,-436){\makebox(0,0)[lb]{\smash{{\SetFigFont{12}{14.4}{\rmdefault}{\mddefault}{\updefault}{$S_f$}%
}}}}
\put(4505,-1336){\makebox(0,0)[lb]{\smash{{\SetFigFont{12}{14.4}{\rmdefault}{\mddefault}{\updefault}{$K(t)$}%
}}}}
\put(3825,-1706){\makebox(0,0)[lb]{\smash{{\SetFigFont{12}{14.4}{\rmdefault}{\mddefault}{\updefault}{$(ECF(\Q(t)[x]/f))(t)$}%
}}}}
\put(4505,-2036){\makebox(0,0)[lb]{\smash{{\SetFigFont{12}{14.4}{\rmdefault}{\mddefault}{\updefault}{$\Q(t)$}%
}}}}
\put(6616,-436){\makebox(0,0)[lb]{\smash{{\SetFigFont{12}{14.4}{\rmdefault}{\mddefault}{\updefault}{${S_f}_i$}%
}}}}
\put(6616,-1046){\makebox(0,0)[lb]{\smash{{\SetFigFont{12}{14.4}{\rmdefault}{\mddefault}{\updefault}{$K_i$}%
}}}}
\put(6621,-2036){\makebox(0,0)[lb]{\smash{{\SetFigFont{12}{14.4}{\rmdefault}{\mddefault}{\updefault}{$\Q$}%
}}}}
\put(3341,-936){\makebox(0,0)[lb]{\smash{{\SetFigFont{12}{14.4}{\rmdefault}{\mddefault}{\updefault}{$\GGal(f) $}%
}}}}
\put(5641,-936){\makebox(0,0)[lb]{\smash{{\SetFigFont{12}{14.4}{\rmdefault}{\mddefault}{\updefault}{$X$}%
}}}}
\put(5791,-1886){\makebox(0,0)[lb]{\smash{{\SetFigFont{12}{14.4}{\rmdefault}{\mddefault}{\updefault}{deg $c$}%
}}}}
\put(7650,-1200){\makebox(0,0)[lb]{\smash{{\SetFigFont{12}{14.4}{\rmdefault}{\mddefault}{\updefault}{$G$}%
}}}}
\end{picture}%

%% file: Galois_Constr.bbl
\providecommand{\bysame}{\leavevmode\hbox to3em{\hrulefill}\thinspace}
\providecommand{\MR}{\relax\ifhmode\unskip\space\fi MR }
\providecommand{\MRhref}[2]{%
  \href{http://www.ams.org/mathscinet-getitem?mr=#1}{#2}
}
\providecommand{\href}[2]{#2}
\begin{thebibliography}{CBFS19}

\bibitem[AK99]{AcKl1}
Vincenzo Acciaro and J{\"u}rgen Kl{\"u}ners, \emph{Computing {A}utomorphisms of
  {A}belian {N}umber {F}ields}, Mathematics of Computation \textbf{68} (1999),
  no.~227, 1179--1186.

\bibitem[CBFS10]{magma217}
J.~J. Cannon, W.~Bosma, C.~Fieker, and A.~Steel (eds.), \emph{Handbook of
  {M}agma functions ({V}2.17)}, Computational Algebra Group, University of
  Sydney, 2010, http://magma.maths.usyd.edu.au.

\bibitem[CBFS17]{magma223}
J.~J. Cannon, W.~Bosma, C.~Fieker, and A.~Steel (eds.), \emph{Handbook of
  {M}agma functions ({V}2.23)}, Computational Algebra Group, University of
  Sydney, 2017, {h}ttp://magma.maths.usyd.edu.au.

\bibitem[CBFS19]{magma225}
J.~J. Cannon, W.~Bosma, C.~Fieker, and A.~Steel (eds.), \emph{Handbook of
  {M}agma functions ({V}2.25)}, Computational Algebra Group, University of
  Sydney, 2019, {h}ttp://magma.maths.usyd.edu.au.

\bibitem[CG05]{Cheze2005}
Guillaume Ch{\`e}ze and Andr{\'e} Galligo, \emph{Four lectures on polynomial
  absolute factorization}, pp.~339--392, Springer Berlin Heidelberg, Berlin,
  Heidelberg, 2005.

\bibitem[CH08]{cannon_holt_class}
J.~J. Cannon and D.~F. Holt, \emph{The transitive permutation groups of degree
  32}, Experimental Mathematics \textbf{17} (2008), 307--314.

\bibitem[FK06]{fieker_klueners_impl}
C.~Fieker and J.~Kl{\"u}ners, \emph{Galois group implementations}, {\sc Magma}
  V2.13 implementation with more recent contributions also from A.-S.
  Elsenhans, 2006.

\bibitem[FK14]{fieker_klueners}
\bysame, \emph{Computation of {G}alois groups of rational polynomials}, London
  Mathematical Society Journal of Computation and Mathematics \textbf{17}
  (2014), no.~1, 141 -- 158.

\bibitem[GG02]{gap}
The GAP~Group, \emph{Gap -- groups, algorithms, and programming, version 4.3},
  2002.

\bibitem[He{\ss{}}04]{antshessautiso}
Florian He{\ss{}}, \emph{An algorithm for computing isomorphisms of algebraic
  function fields}, ANTS VI (D.~Buell, ed.), LNCS, vol. 3076, Springer, 2004,
  pp.~263--271.

\bibitem[HK21]{morikawa}
Jan~E. Holly and David Krumm, \emph{Morikawa's unsolved problem}, The American
  Mathematical Monthly \textbf{128} (2021), no.~3, 214--237.

\bibitem[Hul05]{hulpke_class}
Alexander Hulpke, \emph{Constructing transitive permutation groups}, Journal of
  Symbolic Computation \textbf{39} (2005), no.~1, 1--30. \MR{2168238}

\bibitem[Kl{\"u}97]{Kl2}
J{\"u}rgen Kl{\"u}ners, \emph{{\"U}ber die {B}erechnung von {A}utomorphismen
  und {T}eilk{\"o}rpern {A}lgebraischer {Z}ahlk{\"o}rper}, {D}issertation,
  Technische Universit{\"a}t Berlin, 1997,
  \biburl{http://www.math.tu-berlin.de/{$\sim$}kant/publications/diss/diss\_jk%
.ps.gz}.

\bibitem[Kl{\"u}17]{klueners_equal_groups}
\bysame, \emph{personal communication}, 2017.

\bibitem[KS21]{hilbert}
David Krumm and Nicole Sutherland, \emph{Galois groups over rational function
  fields and explicit {H}ilbert irreducibility}, Journal of Symbolic
  Computation \textbf{103} (2021), 108 -- 126.

\bibitem[Lan83]{hilbert_lang}
Serge Lang, \emph{Fundamentals of {D}iophantine geometry}, Springer-Verlag, New
  York, 1983.

\bibitem[Lan02]{lang_algebra}
S.~Lang, \emph{Algebra}, Springer, 2002.

\bibitem[Mar77]{marcus}
Daniel~A. Marcus, \emph{Number fields}, Springer, 1977.

\bibitem[MM99]{IGT}
G.~Malle and B.-H. Matzat, \emph{Inverse {G}alois theory}, Springer, 1999.

\bibitem[Sta73]{stauduhar}
Richard~P. Stauduhar, \emph{The determination of {G}alois groups}, Mathematics
  of Computation \textbf{27} (1973), 981--996.

\bibitem[Ste89]{stewart_galois}
I.~Stewart, \emph{Galois theory}, Chapman and Hall, 1989.

\bibitem[Sti93]{stichtenoth}
H.~Stichtenoth, \emph{Algebraic function fields and codes}, Springer, 1993.

\bibitem[Sut15]{suth_galois_global}
N.~Sutherland, \emph{Computing {G}alois groups of polynomials (especially over
  function fields of prime characteristic)}, Journal of Symbolic Computation
  \textbf{71} (2015), 73--97.

\bibitem[Tig01]{tignol}
Jean-Pierre Tignol, \emph{Galois' theory of algebraic equations}, World
  Scientific, 2001.

\bibitem[vdW49]{waerden_modern}
B.~L. van~der Waerden, \emph{Modern algebra}, Frederick Ungar Publishing Co.,
  1949.

\end{thebibliography}
